\documentclass[onecolumn]{IEEEtran}
\IEEEoverridecommandlockouts
\usepackage{relsize}
\usepackage{amsmath,bbm}
\usepackage{amsfonts}
\usepackage{amsmath, amsthm, amssymb}
\usepackage{tikz}
\usepackage{pgf, hyperref, cite}
\usepackage{subfig}
\usepackage{relsize}
\usetikzlibrary{arrows,automata}
\usepackage{algorithm}
\usepackage[noend]{algpseudocode}
\usepackage{comment}
\usepackage{graphicx}

\newtheorem{assumption}{Assumption}
\newcommand{\be}{\begin{equation}}
\newcommand{\ee}{\end{equation}}
\newcommand{\req}[1]{(\ref{#1})}

\usepackage{setspace}

\def\ew#1{{{\color{black}#1}}}
\def\fm#1{{{\color{black}#1}}}
\def\re#1{{{\color{black}#1}}}

\newtheorem{theorem}{Theorem}[section]
\newtheorem{lemma}[theorem]{Lemma}

\newcommand{\norm}[1]{\left|\left|#1\right|\right|}

\newlength{\smpagewidth}
\newlength{\smpageheight}

\newcommand{\setleftmargin}[1]{
    \addtolength{\textwidth}{\oddsidemargin}
    \addtolength{\textwidth}{1in}
    \addtolength{\textwidth}{-#1}
    \setlength{\oddsidemargin}{-1in}
    \addtolength{\oddsidemargin}{#1}
    \setlength{\evensidemargin}{\oddsidemargin}
}
\newcommand{\setrightmargin}[1]{
    \setlength{\textwidth}{\smpagewidth}
    \addtolength{\textwidth}{-\oddsidemargin}
    \addtolength{\textwidth}{-1in}
    \addtolength{\textwidth}{-#1}
}
\newcommand{\settopmargin}[1]{
    \addtolength{\textheight}{\topmargin}
    \addtolength{\textheight}{1in}
    \addtolength{\textheight}{\headheight}
    \addtolength{\textheight}{\headsep}
    \addtolength{\textheight}{-#1}
    \setlength{\topmargin}{-1in}
    \addtolength{\topmargin}{-\headheight}
    \addtolength{\topmargin}{-\headsep}
    \addtolength{\topmargin}{#1}
}
\newcommand{\setbottommargin}[1]{
    \setlength{\textheight}{\smpageheight}
    \addtolength{\textheight}{-\topmargin}
    \addtolength{\textheight}{-1in}
    \addtolength{\textheight}{-\footskip}
    \addtolength{\textheight}{-#1}
}

\setleftmargin{.71in}
\setrightmargin{.71in}
\settopmargin{.75in}
\setbottommargin{.77in}

\providecommand{\boldsymbol}[1]{\mbox{\boldmath $#1$}}

\title{Superlinearly Convergent Asynchronous Distributed  Network Newton Method}
\author{Fatemeh Mansoori$^\dag$\thanks{$^\dag$Department of Electrical Engineering and Computer Science, Northwestern University} and Ermin Wei$^\dag$}
\date{}

\begin{document}
\maketitle
\begin{abstract}
The problem of minimizing a sum of local convex objective functions over a networked system captures many important applications and has received much attention in the distributed optimization field. Most of existing work focuses on development of fast distributed algorithms under the presence of a central clock. The only known algorithms with convergence guarantees for this problem in asynchronous setup could achieve either sublinear rate under totally asynchronous setting or linear rate under partially asynchronous setting (with bounded delay). In this work, we built upon existing literature to develop and analyze an asynchronous Newton based approach for solving a penalized version of the problem. We show that this algorithm converges almost surely \ew{with global linear rate and local superlinear rate in expectation.} Numerical studies confirm superior performance against other existing asynchronous methods. 
\end{abstract}
\section{Introduction}\label{sec:intro}
Large scale networks and datasets, coming from applications such as Internet, wireless sensor networks, robotic networks and
large scale machine learning problems, are an integral part of modern technology. One main characteristic of these systems is the lack of centralized access to information due to either communication overhead or the large scale of the network. Therefore, control and optimization algorithms deployed in such networks, should be distributed, relying only on local information and processing \cite{ts86, bo11,bot16,no4, aliLinMorse, we10,mo15}. Distributed optimization algorithms decompose the problem into smaller sub problems that are solved in parallel. These algorithms can run either synchronously or asynchronously. In a synchronous distributed iterative algorithm, the agents will need to have access to a central coordinator/clock and must wait for the slowest to finish  before proceeding to the next iteration. Asynchronous implementations removes these requirements. In asynchronous settings, the agents become active randomly in time and update using partial and local information. In this paper, we will focus on develop\fm{ing} an asynchronous method for solving a penalized  version of the problem of minimizing summation of local convex objective functions.

Various algorithms have been introduced to solve optimization problems in an asynchronous distributed way including primal algorithms \cite{ra9, ra10,du12,bo6,ne11,sr11}, primal-dual algorithms \cite{we13,cha16,ch16,pe16,bi15} and quasi-Newton algorithms \cite{ei16}. While some of these methods have convergence rate guarantees and/or captures more general problem formulation, to the best of our knowledge, there is no asynchronous distributed optimization algorithm with guaranteed superlinear rate of convergence. In particular, the asynchronous gossip algorithm in \cite{ra9} converges almost surely for both convex and non-convex objective functions. In \cite{we13} a general formulation for asynchronous distributed ADMM algorithm has been proposed which converges to the optimum at a rate of $O(\frac{1}{k})$. In \cite{cha16}, \cite{ch16}, the linear convergence rate for an ADMM based method has been guaranteed under the assumption of bounded delays and the specific structure of the objective function. In \cite{pe16} an algorithmic framework has been introduced to find a fix point of a non-expansive operator, for which almost surely convergence and the linear rate have been proven. The recent work in \cite{ei16}, shows the linear convergence rate for a quasi Newton method under the partial asynchrony assumption.  Our contribution is to introduce a totally asynchronous (with arbitrary delay) network Newton algorithm which converges almost surely and  achieves in expectation a quadratic rate of convergence. 

To achieve superlinear rate, we build our algorithm on Newton's method and existing literature on distributed Newton method \cite{we10, mo15,jadbabaie2009distributed}. More specifically, the challenge in designing distributed Newton related method lies in the step to compute an inverse Hessian related quantity. The matrix inversion step can be expensive or impossible to carry out in a distributed way. Therefore, we adopt a common technique to replace Hessian inverse with an approximation \cite{gi5,sh14,bertsekas1983projected}. Our paper is closely related to \cite{mo15,mok15}, where the authors developed Newton based distributed synchronous method for the same penalized problem. Our work is motivated by their approach to approximate Hessian inverse and adapted to the asynchronous setting. In \cite{mo15, mok15}, the authors showed that the synchronous version of network Newton algorithm goes through a quadratic convergence phase. The main difference of our approach lies in the novel asynchronous implementation, which requires very different tools of analysis. Moreover, we present a different stepsize selection criteria which unlike the existing stepsize is independent of the optimal function value that is hard to compute.

The rest of this paper is organized as follows: Section II describes the problem formulation. Section III presents the asynchronous network Newton algorithm. Section IV contains the convergence analysis. Section V presents the simulation results that show the convergence speed improvement of our algorithm compared to the existing methods. Section VI contains the concluding remarks. \\
\noindent\textbf{Basic Notation and Notions:}
A vector is viewed as a column vector. For a matrix $A$, we write $A_{ij}$ to denote the component of $i^{th}$ row and $j^{th}$ column.
For a vector $x$, $x_i$
denotes the $i^{th}$ component of the vector.
We use $x'$ and $A'$ to
denote the transpose of a vector $x$ and a matrix $A$ respectively.
 We use standard Euclidean norm (i.e., 2-norm) unless otherwise noted, i.e., for a vector $x$ in $\mathbb{R}^n$, $\norm{x}=\left(\sum_{i=1}^n x_i^2\right)^{\frac{1}{2}}$.  The notation $\mathbbm{1}$ represents the vector of all $1's$. For a real-valued function $f:\mathbb{R}\rightarrow \mathbb{R}$, the gradient vector and the Hessian
 matrix of $f$ at $x$  are denoted by $\nabla f(x)$ and
 $\nabla^2 f({x})$ respectively.

\section{Problem Formulation}\label{sec:model}
We consider the setup where $n$ agents are connected by an undirected static network and the system wide goal is to collectively solve the following optimization problem:
\begin{equation}
\begin{aligned}
 \min_x\quad \frac{1}{2}x^{T}(I-W)x+\alpha\sum_{i=1}^{n}f_{i}(x_{i})\,,\label{optFormulation}
\end{aligned}
\end{equation}
\noindent where $x:=[x_{1},\,x_{2},\,...,\,x_{n}]^\prime\in\mathbb{R}^{n}$,  \footnote{For representation simplicity, we focus on the case where $x_i$ is in $\mathbb{R}$. Our results in this paper can be easily generalized to higher dimensional case.} matrix $I$ is the identity matrix of size $n$ by $n$, $\alpha>0$ is a positive scalar,  and \ew{the consensus matrix} $W\in\mathbb{\mathbb{R}}^{n\times n}$ is a symmetric nonnegative matrix with the
following properties:
\begin{equation*}
\begin{aligned}
W'& =W,\quad & W\mathbbm{1}&=\mathbbm{1},\quad
\mbox{null}\{I-W\}=\mbox{span}\{1\},\quad & 0\leq& W_{ij}
<1\,.\label{eq:3}
\end{aligned}
\end{equation*}  
Moreover, matrix $W$ represents the network topology, where $W_{ij}\neq 0$ if and only if agents $i$ and $j$ are connected in the underlying network topology. Each function $f_{i}:\mathbb{R}\rightarrow\mathbb{R}$ is twice differentiable and convex. We denote by $\mathcal{N}_{i}$ the set of neighbors of agent $i$ in the underlying network and  $F: \mathbb{R}^n\to\mathbb{R}$ to be the objective function, i.e., 
\be\label{eq:defF}F(x) = \frac{1}{2}x^{T}(I-W)x+\alpha\sum_{i=1}^{n}f_{i}(x_{i}). \ee In this distributed setting, each agent $i\in\{1, 2, ...,n\}$ has access to its local cost function $f_i$, a local decision variable $x_{i}\in\mathbb{R}$ and local positive weights $W_{ij}$ for $j$ in $\mathcal{N}_i$.
This setup models the problems where each agent has a component of the system objective function and aims to
minimize its local cost function, while keeping its variable equal
to those of neighboring agents. 

We study problem \req{optFormulation}, since it can be viewed as a penalized version of a distributed optimization problem, where the objective function is a sum of local convex
cost functions, i.e., 
\begin{equation}
\begin{aligned}
\min_{x}&\quad\sum_{i=1}^{n}f_{i}(x_{i})\,,\\ \mbox{s.t.}&\quad x_{i}=x_{j}\,\,\,\forall\,i,\,j\,\in\mathcal{N}_{i}\,.\label{consensusFormulation}
\end{aligned}
\end{equation}
The term $\frac{1}{2}x^{T}(I-W)x$ in problem \req{optFormulation} corresponds to the penalty on constraint violation in problem \req{consensusFormulation}, since any feasible solution $x = [x_i]_i$ satisfies $Wx = Ix$. The scalar $\alpha$ reflects the weight of objective function relative to penalty on constraint violation. We observe that by varying the value of $\alpha$, we can obtain optimal  solution of \eqref{consensusFormulation}. We here focus on solving \eqref{optFormulation} for one particular value of $\alpha$.
Formulation \req{consensusFormulation} is in turn the distributed formulation of the centralized problem of $\min_x \sum_{i=1}^n f_i(x)$, which is used widely in machine learning, signal processing and sensor networks applications \cite{we13, ch16, mota2013d,bo11}. We will adopt the following assumptions on problem \eqref{optFormulation}.
\begin{assumption}[Bounded Hessian]\label{assm:BoundedHessian} The local objective functions $f_{i}(x)$ are convex,
	twice continuously differentiable with bounded Hessian, i.e. for all $x_i$ in $\mathbb{R}$\[
	0<m\leq\nabla^{2}f_{i}(x_i)\leq M<\infty.\]
\end{assumption}
\begin{assumption} [Lipschitz Hessian]\label{assm:LipHessian} The local objective function Hessian matrices, $\nabla^{2}f_{i}(x_i)$, are L-Lipschitz continuous with respect to the Euclidean norm, i.e., for all $x_i,\,\bar{x}_i$ in $\mathbb{R}$,
\[	\big\Vert \nabla^{2}f_{i}(x_i)-\nabla^{2}f_{i}(\bar{x}_i)\big\Vert \leq L\big\Vert x_i-\bar{x}_i\big\Vert.\]
\end{assumption}
\begin{assumption} [Bounded Consensus Matrix Weight]\label{assm:Consensus} There exist positive scalars $\delta$ and $\Delta$ with $0<\delta\leq \Delta<1$, such that the diagonal elements of the consensus matrix $W$ satisfy
\[	\delta\leq W_{ii}\leq\Delta,\,\,\,\,i=1,\,2,\,...,\,n\,.\]\end{assumption}
The first two assumptions are standard conditions on the local objective functions for developing Newton based algorithms \cite{boydbook}. The first one requires that the eigenvalues of the Hessian matrix are bounded, which implies that the objective functions are $m-$strongly convex. The second assumption states that the Hessian does not change too fast. Both of these assumptions are satisfied by quadratic objective functions. The last assumption is on the matrix $W$, which is satisfied by many standard choices of consensus matrices \cite{nedic2009distributed,ts86,xiao2005scheme}. 
 
For the rest of the paper, we will assume these assumptions hold and our goal is to develop an asynchronous distributed Newton based method to solve problem \req{optFormulation}, which can achieve superlinear rate of convergence.

\section{Asynchronous Network Newton Method}\label{sec:alg}

Our asynchronous method is based on Newton's algorithm for unconstrained problem, which is an iterative algorithm of the following form
\[
x(t+1)=x(t)+\varepsilon d(t),
\]
where $\varepsilon$ is some positive stepsize and $d(t)$ is the Newton direction. The notation $(t)$ indicates the value of the variable at $t^{th}$ iteration. To compute Newton direction, we use $g$ and $H$ to denote the gradient and Hessian of objective function respectively, i.e., $g(t)=\nabla F(x(t))$ and $H(t)=\nabla^{2}F(x(t))$. Then we have Newton direction is given by $d(t)=-H(t)^{-1}g(t)$. 

We use definition of function F [c.f.\ \req{eq:defF}] and have that each component of gradient $g$ is given by
\be\label{eq:gradient} g_i(t)=[(I-W)x(t)]_i+\alpha \nabla f_{i}(x_{i}(t)).\ee The Hessian matrix $H$ can be expressed as
\be \label{eq:defH} H(t)=I-W+\alpha G(t),\ee
 where $G(t)\in\mathbb{R}^{n\times n}$ is a diagonal matrix
with 
\be\label{eq:defG}G_{ii}(t)=\nabla^{2}f_{i}(x_{i}(t)).\ee The main difficulty to develop distributed Newton based methods is to compute Newton direction, which involves Hessian inverse and can not be computed in a distributed way directly. Our asynchronous method adopts the matrix splitting technique in existing literature \cite{cottle1992linear,saad2003iterative,we10, mo15} to compute Newton direction in an distributed way. 

\subsection{Background on Approximation of Newton Direction}
We first summarize the technique used in \cite{mo15} to solve the same problem in a synchronous distributed way, we will then introduce our asynchronous implementation of this algorithm. The main idea is to represent the Hessian inverse as a convergent series of matrices, where each of the term can be computed using local information. The algorithm will then take a finite truncated summation of the terms to approximate the inverse of Hessian matrix. 

We first split the Hessian matrix $H$ [c.f.\ \eqref{eq:defH}] as the following,
\be\label{eq:Hsplit}
H(t)=D(t)-B,\ee
with
\be
D(t)=\alpha G(t)+2(I-W_{d}),
\quad B=I-2W_{d}+W,\label{eq:defDB}\ee
where  $W_{d}$ is a diagonal matrix with $[W_d]_{ii} = W_{ii}$. 
By definition of matrix G [c.f.\ Eq.\ \eqref{eq:defG}] and the assumption that the local functions have bounded second derivative [c.f.\ Assumption \ref{assm:BoundedHessian}], $G(t)$ is a positive definite matrix. By Assumption \ref{assm:Consensus}, we have $[W_d]_{ii}=[W]_{ii}<1$ and thus $I-W_{d}$ is also positive definite. Therefore, the diagonal matrix $D(t)$ is positive definite and thus inVertible. By factoring $D(t)^{1/2}$ on both sides of Eq.\ \eqref{eq:Hsplit}, we can write $H(t)$ as
\[ H(t) = 
D(t)^{1/2}\left(I-D(t)^{-1/2}BD(t)^{-1/2}\right)
D(t)^{1/2},\] which implies that 
\[ H(t)^{-1} = 
D(t)^{-1/2}(I-D(t)^{-1/2}BD(t)^{-1/2})^{-1}
D(t)^{-1/2},\]

We can write the middle inverse term 
\begin{align*}\left(I\right.\left.-D(t)^{-1/2}BD(t)^{-1/2}\right)^{-1} = \sum_{k=0}^\infty \left(D(t)^{-1/2}BD(t)^{-1/2}\right)^k,\end{align*} whenever spectral radius (largest eigenvalue by magnitude) of matrix $D(t)^{-1/2}BD(t)^{-1/2}$ is strictly less than 1 \cite{horn1985matrix}. The particular structure of matrices D and B were chosen, because the following lemma from \cite{mo15} guarantees that the spectral radius of matrix $D(t)^{-1/2}BD(t)^{-1/2}$ is strictly less than 1 for our problem. 
\begin{lemma}  \label{proposition 2} Under Assumptions \ew{\ref{assm:BoundedHessian} and \ref{assm:Consensus}}, $D(t)^{-1/2}BD^{-1/2}$
	is positive semi-definite
	with bounded eigenvalues, i.e., 
	\begin{equation}
	\begin{aligned}
	\boldsymbol{0}\preceq D(t)^{-1/2}BD(t)^{-1/2}\preceq\rho I\,,\label{eq:26}
	\end{aligned}
	\end{equation}
 where $\rho:=2(1-\delta)/(2(1-\delta)+\alpha m)<1$.
\end{lemma}
We can hence write the Hessian inverse as 
\[
H(t)^{-1}=D(t)^{-1/2}\sum_{k=0}^{\infty}\left(D(t)^{-1/2}BD(t)^{-1/2}\right)^{k}D(t)^{-1/2}.
\]
Hence the Newton direction can be represented as 
\begin{align}\label{eq:dCom}& d(t) = -D(t)^{-1/2}\sum_{k=0}^{\infty}\left(D(t)^{-1/2}BD(t)^{-1/2}\right)^{k}D(t)^{-1/2} g(t).\end{align}

Following the same analysis as in \cite{mo15}, we can check the distributed implementation of the above equation. We first note that each of the diagonal elements of $D(t)$ can be computed as 
\begin{equation*}
D_{ii}(t)=\alpha\nabla^{2}f_{i}(x_{i}(t))+2(1-W_{ii}),
\end{equation*}
which can be obtained locally at each node $i$, since both $f_i$ and $W_{ii}$ are local information. Moreover, elements of matrix $B$ satisfy
\[B_{ii}=1-2W_{ii}+W_{ii} = 1-W_{ii},\quad B_{ij}=W_{ij},\]  which can also be computed using local information available to agent $i$. The multiplication by diagonal matrix $D(t)^{-1/2}$ is effectively scaling using local information, which can be implemented locally. Multiplication of matrix $B$ corresponds to communicating with immediate neighbors, which can also be carried out locally. The $k^{th}$ order term in Eq. \eqref{eq:dCom}, can be computed via $k$ local neighborhood information exchanges, i.e., information from neighbors of $k-$hop away. Hence Newton direction $d$ can be computed using local information. However, due to computation limitation, we will truncate the series to include only finitely many terms and form an approximation of the Newton direction, which results in the network Newton algorithm presented in \cite{mo15}. Since our asynchronous algorithm aims to minimize coordination, we will only keep the $0^{th}$ and $1^{st}$ order terms as the basis for asynchronous algorithm development. We denote by $\hat H(t)^{-1}$ the approximation of Hessian inverse using the first two terms of the infinite series, i.e., 
\be\label{eq:defHatH}\hat{H} (t)^{-1}=D(t)^{-1/2}\left[I+D(t)^{-1/2}BD(t)^{-1/2}\right]D(t)^{-1/2},\ee
resulting in Newton direction approximation defined by 
 \be\label{eq:defd}d(t) = -\hat{H}(t)^{-1}g(t).\ee

\subsection{Asynchronous Network Newton}
Based on the preliminaries from the previous section, we can now develop our asynchronous network Newton algorithm. We assume that each agent is associated with a Poisson clock, which ticks according to a Poisson process. The clocks all have the same parameters and are independent from each other. Whenever the Poisson clock ticks, an agent wakes up and updates using local information and information from immediate neighbors to compute its local Newton direction, i.e., $[d(t)]_i$ according to Eq.\ \eqref{eq:defd}. We call this agent {\it active} (activated). We assume that one$-$hop neighbors of the active agent are notified and can perform some basic computations. Since the clocks are Poisson, the probability of having multiple clock ticks at exactly the same time is zero. We assume that the update can be done in a much faster time scale than the clock activations, i.e., the active agent can finish update before another activation happens. When we are only concerned with the total number of updates (instead of total time elapsed), we can equivalently count the number of iterates by increasing the iteration counter by one, whenever any agent is active. In this view, at each iteration, each agent is activated randomly with equal probability, $p=1/n$, and updates its corresponding variable. The algorithm is given as following in Algorithm \ref{async NN}. 
\begin{algorithm}
\caption{Asynchronous Network Newton}\label{async NN}
\begin{algorithmic}[1]
\State Initialization:  
For $i=1, 2, ..., n$, each agent $i$:\\sets $x_{i}(0)=0$, computes $D_{ii}(0), g_i(0), d_i(0), B_{ii}, B_{ij}$:
\[D_{ii}(0)=\alpha\nabla^{2}f_{i}(x_{i}(0))+2(1-W_{ii}),\quad
g_{i}(0)=(1-W_{ii})x_{i}(0)+\alpha\nabla f_{i}(x_{i}(0)),\] 
\[d_{i}^{(0)}(0)=-D_{ii}(0)^{-1}g_{i}(0),\quad
B_{ii}=1-W_{ii}, \quad B_{ij}=W_{ij},\]
and broadcasts $d_i^{(0)}(0)$ to all neighbors, stores received $d_j^{(0)}$, $x_j$ values from neighbors.
\For{$t=1,2,...$}
\State An agent $i\in \left\lbrace 1,2,...,n\right\rbrace$ is active according to its local clock.
\State \re{Active agent $i$ computes the local Newton direction $d_i(t-1)$ using the most recent information from neighbors, $d_j^{(0)}(t-1)$ for $j$ in $\mathcal{N}_i$ as 
\begin{align*}&d_{i}(t-1)=D_{ii}(t-1)^{-1}\bigl[B_{ii}d_{i}^{(0)}(t-1)-g_{i}(t-1)+\sum_{j\in\mathcal{N}_{i}}B_{ij}d_{j}^{(0)}(t-1)\bigr].\end{align*}}
\State Active agent $i$ takes a Newton step and updates its local iterate by
\[x_{i}(t)=x_{i}(t-1)+\varepsilon d_{i}(t-1)\]
\State Active agent updates $D_{ii}(t), g_i(t), d_i^{(0)}(t)$ by  
\[D_{ii}(t)=\alpha\nabla^{2}f_{i}(x_{i}(t))+2(1-W_{ii}),\]
\[g_{i}(t)=(1-W_{ii})x_{i}(t)+\alpha\nabla f_{i}(x_{i}(t))-\sum_{j\in\mathcal{N}_{i}}W_{ij}x_{j}(t-1),\]
\[d_{i}^{(0)}(t)=-D_{ii}(t)^{-1}g_{i}(t)
 \]
\State Active agent $i$ broadcasts $d_{i}^{(0)}(t), x_i(t)$ to its neighbors. 
\State \re{All agents $j\in\mathcal{N}_i$, listen and store received $d_i^{(0)}(t)$ and $x_i(t)$, update $g_j(t)$ and $d_j^{(0)}(t)$ similar to step $7$, and broadcast $d_j^{(0)}(t)$ to their neighbors.} 
\State \re{All inactive agents $l\in\mathcal{N}_j$ passively listen and store received  $d_j^{(0)}(t)$ values from $j\in\mathcal{N}_i$. All other variables remain at their previous values.}
\EndFor
\State \textbf{end for}
\end{algorithmic}
\end{algorithm}

We next analyze the asynchronous feature of our proposed algorithm. We note that using Eq. (\ref{eq:defHatH}) and Eq. (\ref{eq:defd}), the Newton step in our algorithm can be written as \[d(t)=-D(t)^{-1}g(t)-D(t)^{-1}BD(t)^{-1}g(t).\]We denote by $d^{(0)}(t)$ the the Newton direction in which the Hessian matrix is approximated using the $0^{th}$ order term of the Taylor's expansion, i.e., $d^{(0)}(t)=-D(t)^{-1}g(t)$ . Hence the Newton direction for asynchronous network Newton is equal to\[d(t)=D(t)^{-1}\left(Bd^{(0)}(t)-g(t)\right).\]Nothing that $D(t)$ is diagonal and $B$ is representing the underlying graph of the network, the Newton direction for each agent can be written as \[d_{i}(t)=D_{ii}(t)^{-1}\bigl[B_{ii}d_{i}^{(0)}(t)-g_{i}(t)+\sum_{j\in\mathcal{N}_{i}}B_{ij}d_{j}^{(0)}(t)\bigr].
\]where $g_i(t)$ is computed using Eq. (\ref{eq:gradient}).
We observe that when an agent $j$ is not active, it keeps previous values of $D_{jj}, g_j, x_j, d_j^{(0)}$. Formally, if we denote by $\tau_j(t)$,  with $0\leq\tau_{j}(t)\leq t$, the most recent iteration count up to and include $t$ when agent $j$ was active, then we have for agent $j$, which is not active at iteration t, $x_j(t) = x_j(t-1)= x_j(\tau_j(t))$. Our algorithm is {\it totally asynchronous}, in the sense that it does not assume each agent updates at least once within a certain bounded number of iterations \cite{be89}.  

We next verify that the algorithm can be indeed implemented in a distributed way. In this algorithm, each agent $i$ stores the most up to date local information $D_{ii}$ (scalar), $g_i$, $d_i^{(0)}, W_{ii}, W_{ij}. B_{ii}, B_{ij}$ and information $d_j^{(0)}$ and $x_j$ received from neighbor $j$, hence requires access to storage/memory which scales with the degree of the node. Once we have this storage, and the fact that each agent has access to local  gradient and Hessian information  $\nabla f_i$, $\nabla^2 f_i$, then the initialization phase can be done in a distributed way. Moreover at each iteration $t$, the updates at the active agent can be computed through local operations, using the most recently available information from the neighbors and local objective function information. Once an agent finishes its iterate, it broadcasts updated information $d_i^{(0)}$ and $x_i$ to its neighbors. \re{The agent $j\in\mathcal{N}_i$ receives this information from active agent $i$, updates $g_j(t)$ and $d_j^{(0)}(t)$ using the new information and keeps previous values of $D_{jj}(t-1)$ and $x_j(t-1)$. Agent $j\in\mathcal{N}_i$ broadcasts its most recent $d_j^{(0)}(t)$ to its neighbors.} When an agent is not active, we assume that it may still receive information. This can be achieved by maintaining a queue for each neighbor, and when the agent is active, it reads the most recent information from each queue. Thus Algorithm \ref{async NN} may be implemented in a distributed asynchronous way. 
\ew{To implement the asynchronous algorithm, we remark that each agent does not need a counter of the iteration number. Instead, each agent simply needs to maintain the most updated information of itself $D_{ii}(t)$ and its neighbors $j$ in $\mathcal{N}_i$, $d_j^{(0)}(t)$ and $x_j(t)$. One way to implement this is to have some memory allocated for each neighbor and whenever new information is received, the old information is overwritten. We also assume that the clock activation happens on a slower time scale than the agents update, which implies that only one agent is updating at a time.\footnote{This type of asynchronous algorithm is also known as {\it randomized} algorithm.}
}
\section{Convergence Analysis}\label{sec:conv}
In this section, we first present some existing preliminaries in Section \ref{sec:prel}, including some key relations which will then be used to show almost sure \ew{convergence and rate of} convergence of the proposed asynchronous method in Section \ref{sec:convAsync} and establish local quadratic rate of convergence (in expectation) in Section \ref{sec:Quad}. 

\subsection{Preliminaries}\label{sec:prel}
The first four lemmas are adopted from synchronous network Newton method proposed in \cite{mo15}. These lemmas \ew{have been established} in \cite{mo15} \ew{using only} the characteristics of the local objective functions and the consensus matrix $W$ \ew{and are independent of the algorithm implementation. We restate them here for completeness.}
\begin{lemma} \label{lemma 01}  If Assumption \ref{assm:LipHessian} holds, then for every $x, \bar{x}\in\mathbb{R}^{n}$\ew{,} the Hessian matrix, $H(x):=\nabla^{2}F(x)$ , is $\alpha L$-Lipschitz
continuous, i.e.,

\begin{equation*}
\big\Vert H(x)-H(\bar{x})\big\Vert \leq\alpha L\big\Vert x-\bar{x}\big\Vert.
\end{equation*}
\end{lemma}

\begin{lemma} \label{lemma 1} If Assumptions \ref{assm:BoundedHessian},\ref{assm:LipHessian} and \ref{assm:Consensus} hold, the eigenvalues
of $H(t)$ , $D(t)$ , and $B$ [c.f.\ Eqs.\ \eqref{eq:defH}, \eqref{eq:defDB}] are bounded for all $t$ by
\begin{align*}
\alpha mI\preceq & H(t) \preceq(2(1-\delta)+\alpha M)I,\\
(2(1-\Delta)+\alpha m)I&\preceq D(t) \preceq(2(1-\delta)+\alpha M)I,\\
\boldsymbol{0} &\preceq B\preceq 2(1-\delta)I.
\end{align*}
\end{lemma}
\begin{lemma}  \label{lemma:lemma 4.2} Under Assumptions \ref{assm:BoundedHessian} and \ref{assm:Consensus}, the eigenvalues of the approximated Hessian inverse [cf.\ Eq.\ \eqref{eq:defHatH}] are bounded for all $t$ by
\[
\lambda I\preceq\hat{H}(t)^{-1}\preceq\Lambda I\,,
\]
where 
\be\label{eq:defLambda}
\varLambda=\frac{1+\rho}{2(1-\Delta)+\alpha m}, \lambda=\frac{1}{2(1-\delta)+\alpha M}\,\,\,.
\ee
\end{lemma}
The next lemma from \cite{po87} will be used to establish almost sure convergence of the asynchronous network Newton algorithm.
\begin{lemma}  \label{lemma:lemma 4.5} Let $\left(\Omega,\mathcal{\,F},\,\mathcal{P}\right)$
be a probability space and $\mathcal{F}_{0}\subseteq\mathcal{F}_{1}\subseteq...$
be a sequence of sub $\sigma$- fields of $\mathcal{F}$. Let $\left\{ X_{t}\right\} ,\left\{ Y_{t}\right\} ,\left\{ Z_{t}\right\} ,$
and $\left\{ W_{t}\right\} $be $\mathcal{F}_{t}$ -measurable random
variables such that $\left\{X_{t}\right\}$ is bounded below and $\left\{ Y_{t}\right\} $
, $\left\{ Z_{t}\right\} $ , and $\left\{W_{t}\right\}$ are non-negative with $\sum_{t=0}^{\infty}Y_{t}<\infty$
and $\sum_{t=0}^{\infty}W_{t}<\infty$ , if 
\begin{equation}
\begin{aligned}
\mathbb{E}\left[X_{t+1}\Big|\mathcal{F}_{t}\right]\leq(1+Y_{t})X_{t}-Z_{t}+W_{t}\,,\label{eq:29}
\end{aligned}
\end{equation}
then \fm{with probability 1,} $\left\{ X_{t}\right\}$ converges and $\sum_{t=0}^{\infty}Z_{t}<\infty$.
\end{lemma}


The last two lemmas are adopted from \cite{boydbook}, and will be used as key relations in the convergence rate analysis. 

\begin{lemma} \label{lemma:lemma 4.8} If $f:\mathbb{R}^n\to\mathbb{R}$ is a twice continuously differentiable function with
$\iota$-Lipschitz continuous Hessian, then for any $u, v$ in $\mathbb{R}^n$, we have
\begin{equation*}
\begin{aligned}
\big\Vert \nabla f(v)-\nabla f(u)-\nabla^{2}f(u)(v-u)\big\Vert \leq\frac{\iota}{2}\big\Vert v-u\big\Vert^2.
\end{aligned}
\end{equation*}
\end{lemma}

\begin{lemma} \label{functionbound}
If $f:\mathbb{R}^n\to\mathbb{R}$ is a strongly convex function in $\mathbb{R}^n$  with the minimum value of $f^*$ and $\mu I\preceq\nabla^2f(y)\preceq \cal{M}I$, then for any $u, v$ in $\mathbb{R}^n$, we have
\[f(u)\geq f(v)-\frac{1}{2\mu}\big\Vert\nabla f(v)\big\Vert^2,\] and\[f^*\leq f(v)-\frac{1}{2 \cal{M}}\big\Vert\nabla f(v)\big\Vert^2.\]
\end{lemma}
\subsection{Convergence of Asynchronous Network Newton Algorithm}\label{sec:convAsync}
In this section, we show that the asynchronous network Newton algorithm converges to an optimal point almost surely with a global linear rate of convergence in expectation. We will first introduce some notation used to connect asynchronous and synchronous algorithms. At each iteration, we define a random diagonal {\it activation matrix} $\Phi(t)$ in $\mathbb{R}^{n\times n}$ by
\begin{equation}
\begin{aligned}
\Phi(t)_{ii}=\begin{cases}
1 & \mbox{if i is active at time t,}\\
0 & \mbox{otherwise.}
\end{cases}\label{eq:15}
\end{aligned}
\end{equation}
This matrix indicates which agent is active at time t. We also use $\mathcal{F}_{t}$ to denote the $\sigma$-field capturing all realizations  (activations) of the algorithm up to and including time $t$. Then conditioned on $\mathcal{F}_{t}$, we can now rewrite the asynchronous newton direction generated by Algorithm \ref{async NN} by 
\begin{equation}
\begin{aligned}
d_{i}(t-1)=\begin{cases}
-\left[\hat{H}(t-1)^{-1}g(t-1)\right]_{i} & \mbox{if i is active at time t,}\\
0 &\mbox{otherwise.}
\end{cases}
\end{aligned}\label{eq:dt-1}
\end{equation}
The asynchronous network Newton update formula can be aggregated as 
\[
x(t)=x(t-1)+\varepsilon d(t-1)\,.\]
Conditioned on $\mathcal{F}_{t-1}$, we have that $d(t-1)$ is a random variable, given by 
\[d(t-1)=-\Phi(t)\hat{H}(t-1)^{-1}g(t-1),\] where the random matrix $\Phi(t)$ chooses one element of $\hat{H}(t-1)^{-1}g(t-1)$ to keep in $d(t-1)$ and makes the rest 0 as in Eq.\ \eqref{eq:dt-1}. Thus, we have that the asynchronous Newton step can be written as
\be\label{eq:xUpdateAsync}x(t)=x(t-1)-\varepsilon\Phi(t)\hat{H}(t-1)^{-1}g(t-1).\ee
We note that since each agent is active with equal probability, we have that \be\label{eq:expPhi}\mathbb{E}[\Phi(t)|\mathcal{F}_{t-1}] = \frac{1}{n}I.\ee

We next establish almost sure convergence of our asynchronous algorithm.
\begin{theorem}  \label{theorem:theorem 4.4} Consider the iterates $\{x(t)\}$ generated by the asynchronous network Newton algorithm as in Algorithm \ref{async NN}, where the stepsize $\varepsilon$
is chosen as 
\begin{equation}
\begin{aligned}
0<\varepsilon\leq 2\left(\frac{\lambda}{\Lambda}\right)^{2}\,,\label{eq:30}
\end{aligned}
\end{equation}
then the sequence $\left\{ F(x(t))\right\} $ converges to its optimal value, denoted by $F(x^*)$,
almost surely.\label{thm:conve}
\end{theorem}
\begin{proof} Using the Taylor's theorem, we have that for any $a$, $b$ in $\mathbb{R}^n$,
\[
F(a)= F(b)+g(b)'(a-b)+\frac{1}{2}(a-b)' H(c)(a-b),
\] for some $c$ on the line segment between $a$ and $b$. By using the bound on Hessian matrix in Lemma \ref{lemma 1} we have
\[
F(a)\leq F(b)+g(b)'(a-b)+\frac{2(1-\delta)+\alpha M}{2}(a-b)'(a-b).
\]
Thus, for any realization, we can substitute $a=x(t)$, $b=x(t-1)$ and $\lambda=\frac{1}{2(1-\delta)+\alpha M}$ from Eq. \ref{eq:defLambda} and have
\begin{equation}
\begin{aligned}
F(x(t))\leq F(x(t-1))+g(t-1)^\prime(x(t)-x(t-1))+\frac{1}{2\lambda}\big\Vert x(t)-x(t-1)\big\Vert ^{2}\,,\label{eq:31}
\end{aligned}
\end{equation}
From Eq. (\ref{eq:xUpdateAsync}), we have
\begin{equation}
x(t)-x(t-1)=-\varepsilon\Phi(t)\hat{H}(t-1)^{-1}g(t-1)\,.\label{phiterm}
\end{equation}
Taking expectation from both sides of (\ref{eq:31}) conditioned on  $\mathcal{F}_{t-1}$ and using (\ref{phiterm}) we get
\begin{equation*}
\begin{aligned}
&\mathbb{E}\left[F(x(t))\Big|\mathcal{F}_{t-1}\right]\leq F(x(t-1))-\varepsilon g(t-1)^\prime\mathbb{E}\left[\Phi(t)\Big|\mathcal{F}_{t-1}\right]\hat{H}(t-1)^{-1}g(t-1)\\&+\frac{\varepsilon^{2}}{2\lambda}\mathbb{E}\left[\big\Vert \Phi(t)\hat{H}(t-1)^{-1}g(t-1)\big\Vert ^{2}\Big|\mathcal{F}_{t-1}\right]\,.\label{eq:32}
\end{aligned}
\end{equation*}
Note that each agent is active with equal probability from iteration $t-1$ to $t$, we have
\begin{align*}\mathbb{E}&\left[\big\Vert \Phi(t)\hat{H}(t-1)^{-1}g(t-1)\big\Vert ^{2}\Big|\mathcal{F}_{t-1}\right]=\sum_{i=1}^{n}\frac{1}{n}\left[\hat{H}(t-1)^{-1}g(t-1)\right]_{i}^{2}=\frac{1}{n}\big\Vert \hat{H}(t-1)^{-1}g(t-1)\big\Vert ^{2},\label{eq:34}
\end{align*}
where the last equality follows from definition of Euclidean norm.
Using the previous two relations and Eq.(\ref{eq:expPhi}), we have 
\begin{equation*}
\begin{aligned}
\mathbb{E}&\left[F(x(t))\Big|\mathcal{F}_{t-1}\right]\leq F(x(t-1))-\frac{\varepsilon }{n}g(t-1)^\prime\hat{H}(t-1)^{-1}g(t-1)+\frac{\varepsilon^{2}}{2\lambda n}\big\Vert \hat{H}(t-1)^{-1}g(t-1)\big\Vert ^{2}\,.\label{eq:32}
\end{aligned}
\end{equation*}

Using Cauchy-Schwarz inequality and Lemma \ref{lemma:lemma 4.2}, which establishes a bound on the approximated Hessian, we have
\[-\frac{\varepsilon }{n}g(t-1)^\prime\hat{H}(t-1)^{-1}g(t-1)\leq -\frac{\varepsilon \lambda}{n}\norm{g(t-1)}^2,\]
and 
\[\big\Vert \hat{H}(t-1)^{-1}g(t-1)\big\Vert ^{2}\leq \Lambda^2\norm{g(t-1)}^2.\]
Combining the three relations above yields
\begin{equation}
\begin{aligned}
\mathbb{E}\left[F(x(t))\Big|\mathcal{F}_{t-1}\right]\leq F(x(t-1))-\left(\frac{\varepsilon\lambda}{n}-\frac{\varepsilon^{2}\Lambda^{2}}{2n\lambda}\right)\big\Vert g(t-1)\big\Vert ^{2}.\label{eq:martingale}
\end{aligned}
\end{equation}

We next argue that the scalar $\frac{\varepsilon\lambda}{n}-\frac{\varepsilon^{2}\Lambda^{2}}{2n\lambda}\geq 0$. We start by rewriting it as 
\[\frac{\varepsilon\lambda}{n}-\frac{\varepsilon^{2}\Lambda^{2}}{2n\lambda} = \frac{2\varepsilon\lambda^2-\varepsilon^2\Lambda^2}{2n\lambda} = \frac{\varepsilon(2\lambda^2-\varepsilon\Lambda^2)}{2n\lambda}.\] Since the stepsize $\varepsilon$ satisfies the bounds in (\ref{eq:30}), i.e.,
\[\varepsilon\leq 2\left(\frac{\lambda}{\Lambda}\right)^{2},
\]
The scalar $\frac{\varepsilon\lambda}{n}-\frac{\varepsilon^{2}\Lambda^{2}}{2n\lambda}$ is nonnegative. In addition, we have that $F(x(t))$ is strongly convex, thus bounded below by its second order approximation \cite{boydbook} and hence bounded below. Therefore, we can use Eq.\ \ref{eq:martingale} Lemma \ref{lemma:lemma 4.5}, where $Y_t = 0$, $W_t=0$, to conclude that the sequence $\left\{ F(x(t))\right\} $
converges almost surely and $\sum_{t=0}^{\infty}\left(\frac{\varepsilon\lambda}{n}-\frac{\varepsilon^{2}\Lambda^{2}}{2n\lambda}\right)\big\Vert g(t-1)\big\Vert ^{2}<\infty$, 
\fm{with probability 1}, which means that $\big\Vert g(t)\big\Vert ^{2}$ converges to
zero \ew{almost surely}. Combining these two results completes the proof of theorem \ref{theorem:theorem 4.4}.
\end{proof}

\ew{The following theorem establishes global linear rate of convergence in expectation of the asynchronous distributed network Newton algorithm.}

\begin{theorem} \label{linconv}
Consider the iterate $\left\{ x(t)\right\} $ generated by the asynchronous network Newton algorithm as in Algorithm \ref{async NN} with any $x(0)$ in a network with more than one agent. If the stepsize
$\varepsilon$ satisfies 
\begin{equation}
\begin{aligned}
0<\varepsilon< \min\,\left\{ 1,\,2\left(\frac{\lambda}{\Lambda}\right)^{2}\right\} \,,\label{eq:step}
\end{aligned}
\end{equation}
then the sequence $\left\{ F(x(t))\right\}$ converges linearly in expectation to its optimal value, i.e.,
\begin{equation*}
\begin{aligned}
\mathbb{E}\left[F(x(t))-F^*\right]\leq\left(1-\beta\right)^{t}\left[F(x(0))-F^*\right]\,.
\end{aligned}
\end{equation*}
\noindent where $\beta=\frac{\alpha m\varepsilon(2\lambda^{2}-\varepsilon\Lambda^{2})}{n\lambda}$.\\
\end{theorem}

\begin{proof}
We use Eq. (\ref{eq:martingale}) to prove the global linear rate of convergence. Considering the fact that our objective function, $F(x)$, is strongly convex and using the result of Lemma \ref{functionbound}, we have
\begin{equation}
-\big\Vert g(t-1)\big\Vert ^{2}\leq-2\alpha m\left(F(x(t-1))-F^*\right),\label{eq:bound}
\end{equation}
Subtracting $F^*$ from both sides of (\ref{eq:martingale}) and substituting the bound in Eq. (\ref{eq:bound}), we have
\begin{equation}
\begin{aligned}
\mathbb{E}\left[F(x(t))-F^*\Big|\mathcal{F}_{t-1}\right]\leq\left(1-\beta\right)\left(F(x(t-1))-F^*\right)\,,\label{eq:linconv1}
\end{aligned}
\end{equation}
\noindent where $\beta=\frac{\alpha m\varepsilon(2\lambda^{2}-\varepsilon\Lambda^{2})}{n\lambda}$.\\
We next take expectation from both sides of (\ref{eq:linconv1}) with respect to $\mathcal{F}_{t-2}$ , ..., $\mathcal{F}_{0}$. Using the tower rule of expectations we have

\begin{equation}
\begin{aligned}
&\mathbb{E}\left[F(x(t))-F^*\Big|\mathcal{F}_{0}\right]=\mathbb{E}\left[F(x(t))-F^*\right]\leq\left(1-\beta\right)^{t}\left(F(x(0))-F^*\right)\,.\label{eq:lin}
\end{aligned}
\end{equation}
We note that Eq. (\ref{eq:lin}) implies the global linear convergence in expectation only if $0<\beta<1$. We next argue that $0<\beta<1$. We note that if the stepsize $\varepsilon$ satisfies the condition in Eq. (\ref{eq:step}), we have
\[2\lambda^{2}-\varepsilon\Lambda^{2}>0,\] 
thus, $\beta>0$. We now show that $\beta<1$. We first rewrite $\beta$ as
\[\beta=\frac{2\alpha m\varepsilon\lambda^{2}}{n\lambda}-\frac{\alpha m\varepsilon^2\Lambda^{2}}{n\lambda}.\] We note that $\frac{\alpha m\varepsilon^{2}\Lambda^{2}}{n\lambda}>0$ and $\lambda=\frac{1}{2(1-\delta)+\alpha M}$, [c.f. Lemma \ref{lemma:lemma 4.2}]. Therefore, \[\beta<\frac{2\alpha m\varepsilon}{n\left(2(1-\delta)+\alpha M\right)}.\] 
Because $1-\delta>0$, we have $\alpha m<\alpha M+2(1-\delta)$, using this together with the fact that $\varepsilon\leq1$ we
obtain $\beta<\frac{2}{n}$ . Finally, having more than one agent in the network means that $n\geq2$ which implies that $\beta<1$. 

\end{proof}
\subsection{Local Quadratic Rate of Convergence}\label{sec:Quad}
We now proceed to prove local quadratic convergence rate in expectation for the asynchronous network Newton algorithm.
We state four lemmas which will be used to establish the quadratic rate of convergence of the sequence $\left\lbrace \norm{D(t-1)^{1/2}\big(x(t)-x^*\big)}\right\rbrace$ in a specific interval.\\

\begin{lemma} \label{error}
Consider the approximated Hessian inverse defined in Eq. (\ref{eq:defHatH}), if Assumptions \ref{assm:BoundedHessian} and \ref{assm:Consensus} hold, then
\[D(t)^{1/2}(I-\hat{H}(t)^{-1}H(t))=\big(D(t)^{-1/2}BD(t)^{-1/2}\big)^2D(t)^{1/2}.\]
\end{lemma}
\begin{proof}
Using the definition of the Hessian matrix, $H(t)$, and its approximated inverse $\hat{H}(t)^{-1}$ from equations (\ref{eq:Hsplit}) and (\ref{eq:defHatH}), we have
\begin{align*}&I-\hat{H}(t)^{-1}H(t)=I-\Big(D(t)^{-1}+D(t)^{-1}BD(t)^{-1}\Big)\big(D(t)-B\big)=I-\Big(I-D(t)^{-1}B+D(t)^{-1}B-\big(D(t)^{-1}B\big)^2\Big)\\&=(D(t)^{-1}B)^2=D(t)^{-1/2}\big(D(t)^{-1/2}BD(t)^{-1/2}\big)^2D(t)^{1/2}\end{align*}
Multiplying both sides of the previous relation by $D(t)^{1/2}$ from the right, completes the proof.
\end{proof}
\begin{lemma} \label{lemma:lemma 4.99} Let $X$ be a non-negative random variable with $n$ different realizations $X_i$, each happens with probability $\frac{1}{n}$. Then, 
\begin{equation*}
\mathbb{E}\left[X ^{2}\right]\leq n\big(\mathbb{E}\left[X \right]\big)^{2}\,.\label{eq:42}
\end{equation*}
\end{lemma}
\begin{proof}
Note that since $X_i$ is non-negative, we have $\sum_{i=1}^{n}X_{i}{}^{2}\leq\left(\sum_{i=1}^{n}X_{i}\right)^{2}$. Therefore,
\begin{equation*}
\begin{aligned}
\mathbb{E}\left[X^{2}\right]=\frac{1}{n}\sum_{i=1}^{n}X_{i}^{2}\leq\frac{1}{n}\left(\sum_{i=1}^{n}X_{i}\right)^{2}=n\left(\frac{1}{n}\sum_{i=1}^{n}X_{i}\right)^{2}=n\big(\mathbb{E}\left[X\right]\big)^{2},\label{eq:43}
\end{aligned}
\end{equation*} where the last equality follows from the definition of Euclidean norm.
\end{proof}

\begin{lemma}\label{lemma:newexpectnorm}
Consider the random activation matrix $\Phi(t)$, $D(t-1)$, and $B$ defined in Eq. (\ref{eq:15}) and (\ref{eq:defDB}) and recall the definition of $\rho$ from Lemma \ref{proposition 2}, then considering the history of the system upto time $t$, for any $y\in\mathbb{R}^n$ we have
\begin{equation*}
\mathbb{E}\left[\Big\Vert \Big(I-\varepsilon\Phi(t)+\varepsilon\Phi(t)Q(t-1)\Big)y\Big\Vert\Big|\mathcal{F}_{t-1}\right]\leq\Gamma_2 \Vert y\Vert.\label{eq:expectbound}
\end{equation*}
where $Q(t-1)=\big(D(t-1)^{-1/2}BD(t-1)^{-1/2}\big)^2$ and $\Gamma_2=\Big(\frac{n-1+\big(1-\varepsilon+\varepsilon\rho^2\big)^2}{n}\Big)^{1/2}<1$.
\end{lemma}
\begin{proof}  
Note that for each realization of the activation matrix, denoted by $\Phi^i$, we have
\begin{equation*}
\Big[\Big(I-\varepsilon\Phi^i(t)+\varepsilon\Phi^i(t)Q(t-1)\Big)y\Big]_j=\begin{cases}
\Big[\big(I-\varepsilon I+\varepsilon Q(t-1)\big)y\Big]_i & \mbox{if $j=i$},\\
y_j & \mbox{otherwise}.
\end{cases}
\end{equation*}
Hence,
\begin{align*}
\Big\Vert\Big(I-\varepsilon\Phi^i(t)+\varepsilon\Phi^i(t)Q(t-1)\Big)y\Big\Vert^2=\Big[\Big(I-\varepsilon I+\varepsilon Q(t-1)\Big)y\Big]_i^2+\big\Vert
y_{-i}\big\Vert^2,
\end{align*}where $y_{-i}$ is a vector with a zero at $i'th$ element and the rest of its elements are the same as vector $y$.
Taking expectation over all possible realizations of matrix $\Phi(t)$, we obtain
\begin{align*}
&\mathbb{E}\left[\Big\Vert \Big(I-\varepsilon\Phi(t)+\varepsilon\Phi(t)Q(t-1)\Big)y\Big\Vert^2\Big|\mathcal{F}_{t-1}\right]=\sum_{i=1}^n\frac{1}{n}\Big\Vert\Big(I-\varepsilon\Phi^i(t)+\varepsilon\Phi^i(t)Q(t-1)\Big)y\Big\Vert^2\\=&\sum_{i=1}^n\frac{1}{n}\Big(\Big[\Big(I-\varepsilon I+\varepsilon Q(t-1)\Big)y\Big]_i^2+\big\Vert
y_{-i}\big\Vert^2\Big)=\frac{1}{n}\Big\Vert\Big(I-\varepsilon I+\varepsilon Q(t-1)\Big)y\Big\Vert^2+\frac{n-1}{n}\big\Vert
y\big\Vert^2.
\end{align*}
Using triangular inequality,  the result of Lemma \ref{proposition 2} to bound $\big\Vert Q(t-1)\big\Vert$, and the fact that $\varepsilon\leq 1$, we have \[\Big\Vert\Big(I-\varepsilon I+\varepsilon Q(t-1)\Big)y\Big\Vert\leq(1-\varepsilon)\big\Vert y\big\Vert+\varepsilon\rho^2\big\Vert y\big\Vert.\]
Therefore,\[\Big\Vert\Big(I-\varepsilon I+\varepsilon Q(t-1)\Big)y\Big\Vert^2\leq\big(1-\varepsilon+\varepsilon\rho^2\big)^2\big\Vert y\big\Vert^2.\]
We now use Jensen's inequality for expectations to obtain
\[\Bigg(\mathbb{E}\Big[\Big\Vert \Big(I-\varepsilon\Phi(t)+\varepsilon\Phi(t)Q(t-1)\Big)y\Big\Vert\Big|\mathcal{F}_{t-1}\Big]\Bigg)^2  \leq\mathbb{E}\left[\Big\Vert \Big(I-\varepsilon\Phi(t)+\varepsilon\Phi(t)Q(t-1)\Big)y\Big\Vert^2\Big|\mathcal{F}_{t-1}\right]\leq\frac{n-1+\big(1-\varepsilon+\varepsilon\rho^2\big)^2}{n}\big\Vert y\big\Vert^2.\]
\end{proof}
\begin{lemma} \label{recursion} Consider the asynchronous network Newton algorithm as in Algorithm \ref{async NN}, and remember the definition of $D(t-1)$ and $B$, then for any $y\in\mathbb{R}^n$ we have
\[\big\Vert D(t-1)^{1/2}y\big\Vert\leq\Big(1+C_1\big\Vert g(t-2)\big\Vert^{1/2}\Big)\big\Vert D(t-2)^{1/2}y\big\Vert,\]
where $C_1=\Big(\frac{\varepsilon\alpha L \Lambda}{2(1-\Delta)+\alpha m}\Big)^{1/2}$.
\end{lemma}
\begin{proof}
We note that if $\big\Vert D(t-1)^{1/2}y\big\Vert\leq \big\Vert D(t-2)^{1/2}y\big\Vert$, the claim is true because $C_1>0$. Therefore, we consider the case with $\big\Vert D(t-1)^{1/2}y\big\Vert> \big\Vert D(t-2)^{1/2}y\big\Vert$. We next use the Lipschitz property of the Hessian to obtain 
\begin{equation}\big\Vert D(t-1)-D(t-2)\big\Vert=\big\Vert H(t-1)-H(t-2)\big\Vert\leq\alpha L\big\Vert x(t-1)-x(t-2)\big\Vert.  \label{eq:recursion}\end{equation}
where $\alpha L$ is the Lipschitz constant of the Hessian matrix according to the result of Lemma \ref{lemma 01}. We also note that \[\Big\vert y^\prime D(t-1)y - y^\prime D(t-2)y \Big\vert=\big\vert y^{\prime} \big(D(t-1)-D(t-2)\big)y\big\vert\leq\alpha L\big\Vert x(t-1)-x(t-2)\big\Vert\big\Vert y\big\Vert^2.\]
Note that $y^\prime D(t-1)y =\big\Vert D(t-1)^{1/2}y\big\Vert^2$ and  $y^\prime D(t-2)y =\big\Vert D(t-2)^{1/2}y\big\Vert^2$, then using Cauchy Schwarz inequality together with the fact that $\big\Vert D(t-1)^{1/2}y\big\Vert> \big\Vert D(t-2)^{1/2}y\big\Vert$, we have
\begin{equation*}\big\Vert D(t-1)^{-1}y\big\Vert^2\leq\big\Vert D(t-2)^{1/2}y \big\Vert^2+\alpha L\big\Vert x(t-1)-x(t-2)\big\Vert\big\Vert y\big\Vert^2. \end{equation*}
Therefore, 
\begin{equation}\big\Vert D(t-1)^{1/2}y\big\Vert\leq\big\Vert D(t-2)^{1/2}y \big\Vert+\Big(\alpha L\big\Vert x(t-1)-x(t-2)\big\Vert\Big)^{1/2}\big\Vert y\big\Vert.\label{recurbound} \end{equation}
In this step we find an upper bound for $\Vert y\Vert$ in terms of $\Vert D(t-2)^{1/2}y\Vert$. We note that \begin{align*}\mu_{min}\big(D(t-2)^{1/2}\big)\Vert y\Vert\leq\big\Vert D(t-2)^{1/2}y\big\Vert,\end{align*} where $\mu_{min}\big(D(t-2)^{1/2}\big)$ is the minimum eigenvalue of matrix $D(t-2)^{1/2}$. Hence,
\begin{align}\Vert y\Vert\leq \frac{1}{\mu_{min}\big(D(t-2)^{1/2}\big)}\Vert D(t-2)^{1/2}y\Vert\leq \frac{1}{\big(2(1-\Delta)+\alpha m\big)^{1/2}}\big\Vert D(t-2)^{1/2}y\big\Vert.\label{boundNorm}\end{align}
We next combine relations (\ref{recurbound}) and (\ref{boundNorm}) to obtain 
\begin{align*}
\big\Vert D(t-1)^{1/2}y\big\Vert\leq\Bigg(1+\Bigg(\frac{\alpha L\big\Vert x(t-1)-x(t-2)\big\Vert}{2(1-\Delta)+\alpha m}\Bigg)^{1/2}\Bigg)\big\Vert D(t-2)^{1/2}y\big\Vert
\end{align*}
Finally, we use the asynchronous network Newton iteration defined in Eq. (\ref{eq:xUpdateAsync}) to substitute $x(t-1)-x(t-2)$ with $-\varepsilon\Phi(t-1)\hat{H}(t-2)^{-1}g(t-2)$ to get \[\big\Vert D(t-1)^{1/2}y\big\Vert\leq\Bigg(1+\Big(\frac{\varepsilon\alpha L}{2(1-\Delta)+\alpha m}\Big)^{1/2}\big\Vert\Phi(t-1)\hat{H}(t-2)^{-1}g(t-2)\big\Vert^{1/2}\Bigg)\big\Vert D(t-2)^{-1}y\big\Vert.\]
Cauchy Schwarz inequality and the facts that $\big\Vert\hat{H}(t-2)^{-1}\big\Vert\leq\Lambda$ [c.f. Lemma \ref{lemma:lemma 4.2}] and $\big\Vert\Phi(t-1)\big\Vert=1$ for all realizations, complete the proof.
\end{proof}
\begin{lemma} \label{lemma:lemma 4.11} Consider the asynchronous network Newton algorithm as in Algorithm \ref{async NN} with the positive stepsize of $\varepsilon\leq 1$,and recall the definition of $0<\beta<1$ from Theorem \ref{linconv} , $\rho<1$ from Lemma \ref{proposition 2}, $C_1$ from Lemma \ref{recursion}, and $\Gamma_2$ from Lemma \ref{lemma:newexpectnorm}, then the sequence of the iterate errors $\left\lbrace x(t)-x^*\right\rbrace$ satisfies
\begin{equation*}
\begin{aligned}
\mathbb{E}\Big[\big\Vert D(t-1)^{1/2}\big(x(t)-x^*\big)\big\Vert\Big]\leq \Gamma_1\Bigg(\mathbb{E}\Big[\big\Vert D(t-2)^{1/2}\big(x(t-1)-x^*\big)\big\Vert\Big]\Bigg)^2+\Gamma(t)\mathbb{E}\Big[\big\Vert D(t-2)^{1/2}\big(x(t-1)-x^*\big)\big\Vert\Big], 
\end{aligned}
\end{equation*}
\noindent where $\Gamma_{1}=\frac{n\big(2(1-\delta)+\alpha M\big)^{1/2}\alpha L\varepsilon \Lambda}{2\big(2(1-\Delta)+\alpha m\big)}$ and $\Gamma(t)=\Gamma_2\Big(1+C_2(1-\beta)^{\frac{t-2}{4}}\Big)$ with \[C_2=C_1 \Big(\frac{2n^2}{\lambda}\big(F(x(0))-F^*\big)\Big)^{1/4}\]
\end{lemma}
\begin{proof} 
By adding and subtracting $x(t-1)$ and $\varepsilon\Phi(t)\hat{H}(t-1)^{-1}H(t-1)\left(x(t-1)-x^*\right)$ from $x(t)-x^*$ we have
\begin{equation*}
\begin{aligned}
&x(t)-x^*=x(t)-x(t-1)+\varepsilon\Phi(t)\hat{H}(t-1)^{-1}H(t-1)\left(x(t-1)-x^*\right)+x(t-1)-x^*\\&-\varepsilon\Phi(t)\hat{H}(t-1)^{-1}H(t-1)\left(x(t-1)-x^*\right).
\end{aligned}
\end{equation*}
We next substitute $x(t)-x(t-1)$  with $-\varepsilon \Phi(t)\hat{H}(t-1)^{-1}g(t-1)$ and add and subtract $\varepsilon\Phi(t)\big(x(t-1)-x^*\big)$ and obtain
\begin{align*} &x(t)-x^*=\varepsilon\Phi(t)\hat{H}(t-1)^{-1}\Big(H(t-1)\big(x(t-1)-x^*\big)-g(t-1)\Big)+\big(I-\varepsilon\Phi(t)\big)\big(x(t-1)-x^*\big)\\&+\varepsilon\Phi(t)\big(I-\hat{H}(t-1)^{-1}H(t-1)\big)\big(x(t-1)-x^*\big).\end{align*}
Multiplying both sides of the previous equality by the diagonal matrix $D(t-1)^{1/2}$ and using the result of Lemma \ref{error} to substitute $D(t-1)^{1/2}\big(I-\hat{H}(t-1)^{-1}H(t-1)\big)$ with $\big(D(t-1)^{-1/2}BD(t-1)^{-1/2}\big)^2D(t-1)^{1/2}$, we have
\begin{align*}&D(t-1)^{1/2}\big(x(t)-x^*\big)=\varepsilon D(t-1)^{1/2}\Phi(t)\hat{H}(t-1)^{-1}\Big(H(t-1)\big(x(t-1)-x^*\big)-g(t-1)\Big)+\\&\big(I-\varepsilon\Phi(t)\big)D(t-1)^{1/2}\big(x(t-1)-x^*\big)+\varepsilon\Phi(t)\big(D(t-1)^{-1/2}BD(t-1)^{-1/2}\big)^2D(t-1)^{1/2}\big(x(t-1)-x^*\big).\end{align*}
We then take norms on both sides and use triangular and Cauchy Schwarz inequalities to obtain
\begin{equation*}
\begin{aligned}
&\big\Vert D(t-1)^{1/2}\big(x(t)-x^*\big)\big\Vert\leq\varepsilon\big\Vert D(t-1)^{1/2}\Phi(t)\hat{H}(t-1)^{-1}\big\Vert\big\Vert H(t-1)\big(x(t-1)-x^*\big)-g(t-1)\big\Vert+\\&\Big\Vert\Big( I-\varepsilon\Phi(t)+\varepsilon\Phi(t)\big(D(t-1)^{-1/2}BD(t-1)^{-1/2}\big)^2\Big) D(t-1)^{1/2}\big(x(t-1)-x^*\big)\Big\Vert.
\end{aligned}
\end{equation*}
Applying the result of Lemma \ref{lemma:lemma 4.8} with $v=x^*$ and $u=x(t-1)$ and considering the fact that $\nabla F(x^*)=0$, yields 
\begin{align*}&\left\Vert H(t-1)\big(x(t-1)-x^*\big)-g(t-1)\right\Vert\leq\frac{\alpha L}{2}\left\Vert x(t-1)-x^*\right\Vert^2.\end{align*} where $\alpha L$ is the Lipschitz constant of the Hessian matrix according to the result of Lemma \ref{lemma 01}. 
Employing the above inequality together with the upper bounds on $\Vert D(t-1)^{1/2}\Vert$  and $\Vert\hat{H}(t-1)^{-1}\Vert$ [c.f. Lemma \ref{lemma 1} and Lemma \ref{lemma:lemma 4.2}], and considering the fact that for every realization of the random activation matrix we have $\Vert\Phi(t)\Vert=1$, we obtain \begin{equation*}
\begin{aligned}
&\big\Vert D(t-1)^{1/2}\big(x(t)-x^*\big)\big\Vert\leq\frac{\big(2(1-\delta)+\alpha M\big)^{1/2}\alpha L\varepsilon \Lambda}{2}\big\Vert x(t-1)-x^*\big\Vert^2+\\&\Big\Vert\Big( I-\varepsilon\Phi(t)+\varepsilon\Phi(t)\big(D(t-1)^{-1/2}BD(t-1)^{-1/2}\big)^2\Big) D(t-1)^{1/2}\big(x(t-1)-x^*\big)\Big\Vert.
\end{aligned}
\end{equation*}
This inequality holds for any random activation of the agents. We now note that conditioned on $\mathcal{F}_{t-1}$, matrix $\Phi(t)$  and $x(t)$ are random variables and $x(t-1)$ is deterministic, we can hence take expectation from both sides of the above inequality and have
\begin{equation*}
\begin{aligned}
&\mathbb{E}\Big[\big\Vert D(t-1)^{1/2}\big(x(t)-x^*\big)\big\Vert\Big|\mathcal{F}_{t-1}\Big]\leq\frac{\big(2(1-\delta)+\alpha M\big)^{1/2}\alpha L\varepsilon \Lambda}{2}\big\Vert x(t-1)-x^*\big\Vert^2+\\&\mathbb{E}\Big[\Big\Vert\Big( I-\varepsilon\Phi(t)+\varepsilon\Phi(t)\big(D(t-1)^{-1/2}BD(t-1)^{-1/2}\big)^2\Big) D(t-1)^{1/2}\big(x(t-1)-x^*\big)\Big\Vert\Big|\mathcal{F}_{t-1}\Big].
\end{aligned}
\end{equation*}
We now consider the result of Lemma \ref{lemma:newexpectnorm} with $y=D(t-1)^{1/2}\big(x(t-1)-x^*\big)$ to obtain
\begin{equation*}
\begin{aligned}
&\mathbb{E}\Big[\big\Vert D(t-1)^{1/2}\big(x(t)-x^*\big)\big\Vert\Big|\mathcal{F}_{t-1}\Big]\leq\frac{\big(2(1-\delta)+\alpha M\big)^{1/2}\alpha L\varepsilon \Lambda}{2}\big\Vert x(t-1)-x^*\big\Vert^2+\\& \Gamma_2\big\Vert D(t-1)^{1/2}\big(x(t-1)-x^*\big)\big\Vert.
\end{aligned}
\end{equation*}
We also have
\[\big\Vert x(t-1)-x^*\big\Vert\leq\frac{1}{\mu_{min}\big(D(t-2)^{1/2}\big)}\big\Vert D(t-2)^{1/2}\big(x(t-1)-x^*\big)\big\Vert,\] hence, \[\big\Vert x(t-1)-x^*\big\Vert^2\leq\frac{1}{2(1-\Delta)+\alpha m}\big\Vert D(t-2)^{1/2}\big(x(t-1)-x^*\big)\big\Vert^2\]
In order to have a recursion we use the result of Lemma \ref{recursion} to bound $\big\Vert D(t-1)^{1/2}\big(x(t-1)-x^*\big)\big\Vert$ in terms of $\big\Vert D(t-2)^{1/2}\big(x(t-1)-x^*\big)\big\Vert$ and obtain
\begin{equation}
\begin{aligned}
&\mathbb{E}\Big[\big\Vert D(t-1)^{1/2}\big(x(t)-x^*\big)\big\Vert\Big|\mathcal{F}_{t-1}\Big]\leq\frac{\big(2(1-\delta)+\alpha M\big)^{1/2}\alpha L\varepsilon \Lambda}{2\big(2(1-\Delta)+\alpha m\big)}\big\Vert D(t-2)^{1/2}\big(x(t-1)-x^*\big)\big\Vert^2+\\& \Gamma_2\Big(1+C_1\big\Vert g(t-2)\big\Vert^{1/2}\Big)\big\Vert D(t-2)^{1/2}\big(x(t-1)-x^*\big)\big\Vert. \label{eq:condexp}
\end{aligned}
\end{equation}
We now consider iteration $t$, taking expectations from both sides of Eq. (\ref{eq:condexp}) and using the law of total expectation yield 
\begin{equation}
\begin{aligned}
&\mathbb{E}\Big[\big\Vert D(t-1)^{1/2}\big(x(t)-x^*\big)\big\Vert\Big]\leq\frac{\big(2(1-\delta)+\alpha M\big)^{1/2}\alpha L\varepsilon \Lambda}{2\big(2(1-\Delta)+\alpha m\big)}\mathbb{E}\Big[\big\Vert D(t-2)^{1/2}\big(x(t-1)-x^*\big)\big\Vert^2\Big]+\\& \Gamma_2\mathbb{E}\Big[\Big(1+C_1\big\Vert g(t-2)\big\Vert^{1/2}\Big)\big\Vert D(t-2)^{1/2}\big(x(t-1)-x^*\big)\big\Vert\Big]. \label{expineq}
\end{aligned}
\end{equation}
The second expected value in the right hand side of the previous inequality is equal to
\begin{align*}\mathbb{E}\Big[\big\Vert D(t-2)^{1/2}\big(x(t-1)-x^*\big)\big\Vert\Big]+ C_1 \mathbb{E}\Big[\big\Vert g(t-2)\big\Vert^{1/2}\big\Vert D(t-2)^{1/2}\big(x(t-1)-x^*\big)\big\Vert\Big]\end{align*}
We note that the Cauchy Schwarz inequality in the context of the expectation states that for any two random variables $X$ and $Y$ such that $\mathbb{E}[X]$, $\mathbb{E}[Y]$, and $\mathbb{E}[XY]$ exist, we have 
\[\Big(\mathbb{E}[XY]\Big)^2\leq\mathbb{E}\big[X^2\big]\mathbb{E}\big[Y^2\big].\]
Therefore,
\[\mathbb{E}\Big[\big\Vert g(t-2)\big\Vert^{1/2}\big\Vert D(t-2)^{1/2}\big(x(t-1)-x^*\big)\big\Vert\Big]\leq\Bigg(\mathbb{E}\Big[\big\Vert g(t-2)\big\Vert\Big]\Bigg)^{1/2}\Bigg(\mathbb{E}\Big[\big\Vert D(t-2)^{1/2}\big(x(t-1)-x^*\big)\big\Vert^2\Big]\Bigg)^{1/2}\]
We next use the result of Lemma \ref{functionbound} on the properties of strongly convex functions together with the fact that $H(t)\preceq \big(2(1-\delta)+\alpha M\big)I=\frac{1}{\lambda}I$, to find an upper bound for $\mathbb{E}\Big[\big\Vert g(t-2)\big\Vert\Big]$ as follows
\[\big\Vert g(t-2)\big\Vert\leq\Big(\frac{2}{\lambda}\Big(F\big(x(t-2)\big)-F^{*}\Big)\Big)^{1/2}.\]
Taking expectation from both sides of the above inequality and using the Jensen's inequality for concave functions together with the linear convergence result from Theorem \ref{linconv}, we have
\[\mathbb{E}\Big[\big\Vert g(t-2)\big\Vert\Big]
\leq\mathbb{E}\Big[\Big(\frac{2}{\lambda}\Big(F\big(x(t-2)\big)-F^{*}\Big)\Big)^{1/2}\Big]\leq\Big(\frac{2}{\lambda}\Big)^{1/2}\Big(\mathbb{E}\big[F\big(x(t-2)\big)-F^{*}\big]\Big)^{1/2}\leq\Big(\frac{2}{\lambda}\Big)^{1/2}(1-\beta)^{\frac{t-2}{2}}\big(F\big(x(0)\big)-F^{*}\big)^{1/2}\]
We also note that considering the result of Lemma \ref{lemma:lemma 4.99} we have
\[\mathbb{E}\Big[\big\Vert D(t-2)^{1/2}\big(x(t-1)-x^*\big)\big\Vert^2\Big]\leq n \Bigg(\mathbb{E}\Big[\big\Vert D(t-2)^{1/2}\big(x(t-1)-x^*\big)\big\Vert\Big]\Bigg)^{2}\]
Using the three previous inequalities together with Eq. (\ref{expineq}) and applying the result of Lemma \ref{lemma:lemma 4.99} yields
\begin{equation*}
\begin{aligned}
&\mathbb{E}\Big[\big\Vert D(t-1)^{1/2}\big(x(t)-x^*\big)\big\Vert\Big]\leq \frac{n\big(2(1-\delta)+\alpha M\big)^{1/2}\alpha L\varepsilon \Lambda}{2\big(2(1-\Delta)+\alpha m\big)}\Bigg(\mathbb{E}\Big[\big\Vert D(t-2)^{1/2}\big(x(t-1)-x^*\big)\big\Vert\Big]\Bigg)^2+\\& \Gamma_2\Bigg(1+C_1\Big(\frac{2n^2(1-\beta)^{t-2}}{\lambda}\big(F(x(0))-F^*\big)\Big)^{1/4}\Bigg)\mathbb{E}\Big[\big\Vert D(t-2)^{1/2}\big(x(t-1)-x^*\big)\big\Vert\Big]. \label{expineqnew} 
\end{aligned}
\end{equation*}
\end{proof}
Before proceeding to the last theorem, we show that the sequence $\left\lbrace\big\Vert D(t-1)^{1/2}\big(x(t)-x^*\big)\big\Vert\right\rbrace$ converges linearly in expectation. 
\begin{lemma}\label{linconvX}
Consider the asynchronous network Newton iterate as in Algorithm \ref{async NN},   if the stepsize $\varepsilon$ satisfies Eq.(\ref{eq:step}), then The sequence $\left\lbrace\big\Vert D(t-1)^{1/2} \big(x(t)-x^*\big)\big\Vert\right\rbrace$ converges linearly in expectation.
\end{lemma}
\begin{proof}
Using the Taylor's theorem and the the strong convexity of the objective function $F(x(t))$, we have 
\[F(x(t))\geq F^*+g(x^*)(x(t)-x^*)+\frac{\alpha m}{2}\big\Vert x(t)-x^*\big\Vert^2,\]where $\alpha m$ is the lower bound on the eigenvalues of $H(t)$, [c.f. Lemma \ref{lemma 1}].  We note that $g(x^*)=0$, therefore \[\big\Vert x(t)-x^*\big\Vert^2\leq\frac{2}{\alpha m}\big(F(x(t))-F^*)\big).\] Multiplying both sides by $\big\Vert D(t-1)^{1/2}\big\Vert$ and using the Cauchy Schwarz inequality, we obtain \[\big\Vert D(t-1)^{1/2}\big(x(t)-x^*\big)\big\Vert^2\leq\big\Vert D(t-1)^{1/2}\big\Vert^2\big\Vert x(t)-x^*\big\Vert^2\leq\frac{2\big(2(1-\delta)+\alpha M \big)}{\alpha m}\big(F(x(t))-F^*\big).\]
We next take expectation from both sides of the previous inequality and apply the result of Lemma \ref{linconv} to obtain
\[\mathbb{E}\Big[\big\Vert D(t-1)^{1/2}\big(x(t)-x^*\big)\big\Vert^2\Big]\leq\frac{2\big(2(1-\delta)+\alpha M \big)}{\alpha m}\mathbb{E}\Big[\big(F(x(t))-F^*)\big)\Big]\leq\frac{2\big(2(1-\delta)+\alpha M \big)(1-\beta)^t}{\alpha m}\big(F(x(0))-F^*\big).\] Employing the Jensen's inequality for expectations yields \[\Bigg(\mathbb{E}\Big[\big\Vert D(t-1)^{1/2}\big(x(t)-x^*\big)\big\Vert\Big]\Bigg)^2\leq\mathbb{E}\Big[\big\Vert D(t-1)^{1/2}\big(x(t)-x^*\big)\big\Vert^2\Big]\leq\frac{2\big(2(1-\delta)+\alpha M \big)(1-\beta)^t}{\alpha m}\big(F(x(0))-F^*\big).\] Taking square root from both sides of the previous relation completes the proof.

\end{proof}
\begin{theorem}\label{localquad}
Consider the asynchronous network Newton iterate as in Algorithm \ref{async NN} and recall the definition of $\Gamma_1$ and $\Gamma(t)$ from Lemma \ref{lemma:lemma 4.11}, then for all $t$ with
 \begin{align}t> \frac{4\ln{\frac{1-\Gamma_2}{C_2\Gamma_2}}}{\ln{(1-\beta)}}+2, \label{tbound}\end{align}
there exists  $0<\theta<\frac{1-\Gamma(t)}{\Gamma_1\Gamma(t)}$, such that the sequence of $\mathbb{E}\left[\big\Vert D(t-1)^{1/2}\big(x(t)-x^*\big)\big\Vert\right]$ satisfies
\begin{equation}
\theta\Gamma(t)\leq\mathbb{E}\left[\big\Vert D(t-1)^{1/2}\big(x(t)-x^*\big)\big\Vert\right]<\frac{\theta}{\theta\Gamma_1+1}, \label{eq:condition}
\end{equation}
and decreases with a quadratic rate in expectation in this interval.
\end{theorem}
\begin{figure}[thpb]
\vspace{-.1cm}
    \centering
    \includegraphics[width=.45\textwidth]{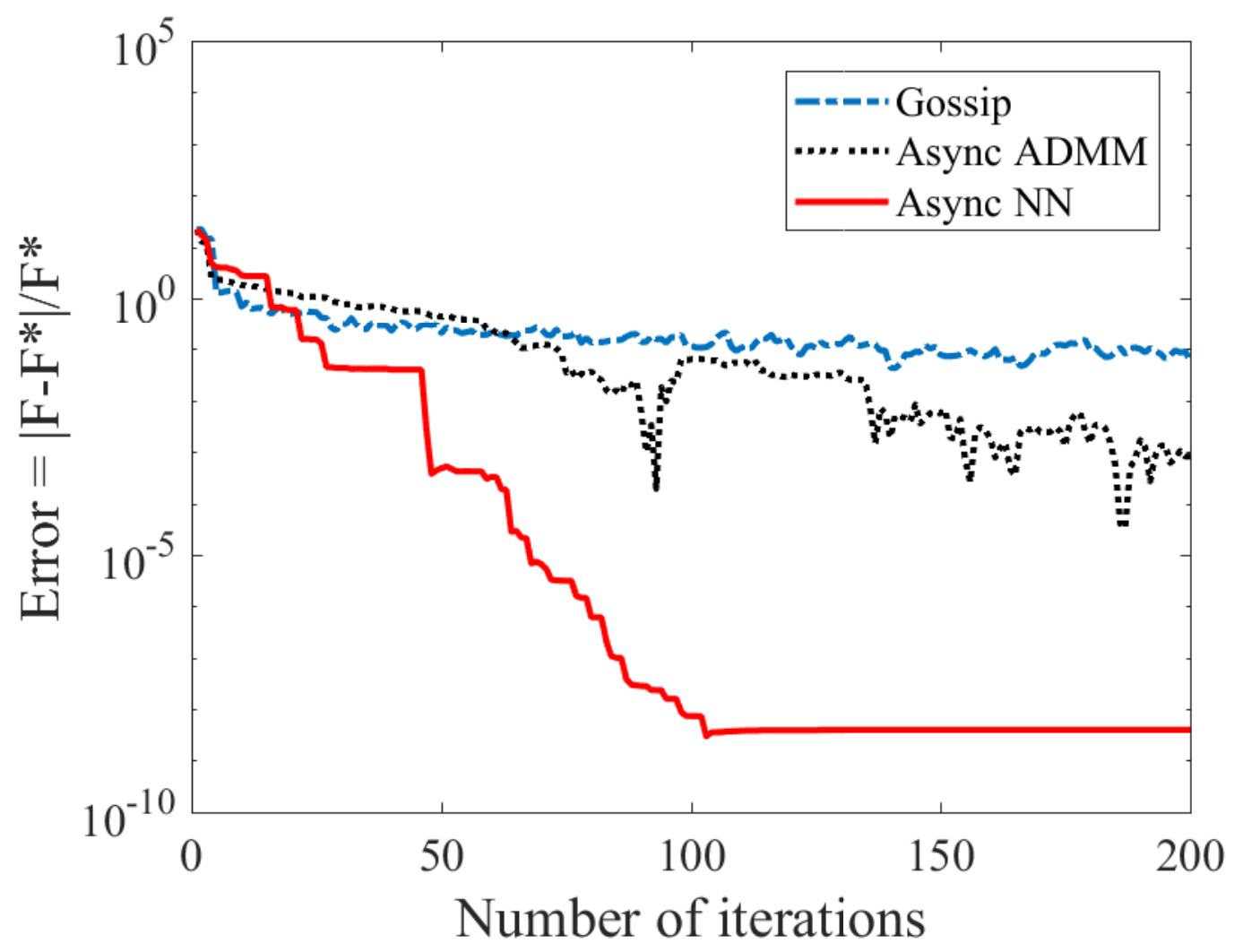}
\caption{Convergence of asynchronous network Newton, asynchronous ADMM and Gossip algorithms in terms of number of iterations}
      \label{fig1}
   \end{figure}
\begin{proof}
We note that in order for the $\theta$ neighborhood to be nonempty, we need $\Gamma(t)<1$. In what follows we show that  for all $t$ satisfying Eq. (\ref{tbound}), $\Gamma(t)<1$. \\Considering the definition of $\Gamma(t)$ from Lemma \ref{lemma:lemma 4.11} , we have
\[\Gamma(t)=\Gamma_2\Big(1+C_2(1-\beta)^{\frac{t-2}{4}}\Big).\]
To have $\Gamma(t)<1$, we need
\[1+C_2(1-\beta)^{\frac{t-2}{4}}<\frac{1}{\Gamma_2},\] therefore, \[(1-\beta)^{\frac{t-2}{4}}<\frac{1-\Gamma_2}{C_2\Gamma_2}.\]
Taking logarithm to the base $1-\beta<1$ from both sides of the above inequality flips the direction of the inequality and results in a lower bound for $t$ as \[t>4\log_{1-\beta}^{\frac{1-\Gamma_2}{C_2\Gamma_2}}+2,\] which is equal to the lower bound in Eq. (\ref{tbound}) by changing the base of the logarithm. We also note that using the definition of $C_2$, the lower bound for $\bar{t}$ will be a function of $F(x(0))-F^*$, which suggests that the further from $F^*$ the algorithm starts, the longer it takes to reach the phase with the quadratic convergence rate.  
We emphasize that since the sequence $\left\lbrace\big\Vert D(t-1)^{1/2}\big(x(t)-x^*\big)\big\Vert\right\rbrace$ converges linearly in expectation, the given neighborhood in Eq. (\ref{eq:condition}) is not empty. We next show that within this interval, the sequence $\left\lbrace \big\Vert D(t-1)^{1/2}\big(x(t)-x^*\big)\big\Vert\right\rbrace$ decreases with a quadratic rate.\\
For analysis simplicity we denote by $[\bar{t},\bar{t}+l]$ the interval in which Eq. (\ref{eq:condition}) is satisfied. Using the result of Lemma \ref{lemma:lemma 4.11} we have \[\mathbb{E}\left[\big\Vert D(\bar{t})^{1/2}\big(x(\bar{t}+1)-x^*\big)\big\Vert\right]\leq\Gamma_1 \left(\mathbb{E}\left[\big\Vert D(\bar{t}-1)^{1/2}\big(x(\bar{t})-x^*\big)\big\Vert\right]\right)^2+\Gamma(t)\mathbb{E}\left[\big\Vert D(\bar{t}-1)^{1/2}\big(x(\bar{t})-x^*\big)\big\Vert\right].\]
We now use the left hand side of Eq. (\ref{eq:condition}) to bound $\Gamma(t)$ and obtain
\[\mathbb{E}\Big[\big\Vert D(\bar{t})^{1/2}\big(x(\bar{t}+1)-x^*\big)\big\Vert\Big]\leq\big(\Gamma_1+\frac{1}{\theta}\big)\Big(\mathbb{E}\Big[\big\Vert D(\bar{t}-1)^{1/2}\big(x(\bar{t})-x^*\big)\big\Vert\Big]\Big)^2.\]
Multiplying both sides of the previous inequality by $\frac{\theta\Gamma_1+1}{\theta}$, we have
\[\frac{\theta\Gamma_1+1}{\theta}\mathbb{E}\Big[\big\Vert D(\bar{t})^{1/2}\big(x(\bar{t}+1)-x^*\big)\big\Vert\Big]\leq\left(\frac{\theta\Gamma_1+1}{\theta}\mathbb{E}\Big[\big\Vert D(\bar{t}-1)^{1/2}\big(x(\bar{t})-x^*\big)\big\Vert\Big]\right)^2.\]
Applying this recursively upto any time $r\in[\bar{t},\bar{t}+l]$, we find that
\begin{equation}
\mathbb{E}\Big[\big\Vert D(r-1)^{1/2}\big(x(r)-x^*\big)\big\Vert\Big]\leq\frac{\theta}{\theta\Gamma_1+1}\Big(\frac{\theta\Gamma_1+1}{\theta}\mathbb{E}\Big[\big\Vert D(\bar{t}-1)^{1/2}\big(x(\bar{t})-x^*\big)\big\Vert\Big]\Big)^{2^{r-\bar{t}}}.\label{eq:quadx}
\end{equation}
We note that the right hand side of Eq. (\ref{eq:condition}) implies that 
$\frac{\theta\Gamma_1+1}{\theta}\mathbb{E}\Big[\big\Vert D(\bar{t}-1)^{1/2}\big(x(\bar{t})-x^*\big)\big\Vert\Big]<1$, hence Eq. (\ref{eq:quadx}) establishes the quadratic convergence rate for all $r\in[\bar{t},\bar{t}+l]$. 
\end{proof} 

\section{Simulation Results}\label{sec:sim}
In this section, we present some numerical studies, where we compare the performance of asynchronous network Newton with asynchronous ADMM and asynchronous Gossip algorithms presented in \cite{we13} and \cite{ra10}. We consider a network of 5 agents each having access to a local objective function of the form $f_{i}(x_{i})=(x_{i}-i)^{2}$ , $i\in\{1,...,5\}$. The underlying network is a complete graph with graph Laplacian matrix $L$. We set the consensus matrix $W$ to be $W=I-\frac{1}{8}L$.
For our asynchronous network Newton we choose $\alpha=1$, stepsize $\varepsilon=0.8$, satisfying the condition in Theorem \ref{thm:conve}. We plot the resulting relative error in objective function value, $\frac{F(x(t))-F^*}{F^*}$, in logarithmic scale in Fig. \ref{fig1} . Asynchronous network Newton is the solid red line, asynchronous Gossip algorithm is the blue dash line and asynchronous ADMM is the black dotted line. We can see clearly that asynchronous network Newton outperforms the other two algorithms, which is expected due to the local quadratic rate. We have also simulated other objective function values and obtained similar results. \\
It's important to note that the optimal value of problems (\ref{optFormulation}) and (\ref{consensusFormulation}) are different. The difference between these two values depends on the penalty coefficient $\frac{1}{\alpha}$. Decreasing $\alpha$ results in a smaller gap between the two optimal values. In this simulation we evaluate the performance of the asynchronous algorithms in reaching their respective optimal values. 
\section{Conclusion}\label{sec:con}
This paper presents an asynchronous distributed network Newton algorithm, which uses matrix splitting techniques to approximate the Hessian inverse and compute Newton step. We also show that the proposed method converges almost surely with in expectation global linear and local quadratic rate of convergence . Simulation results show the convergence speed improvement of the asynchronous network Newton compared to asynchronous ADMM and asynchronous gossip algorithm. Possible future work includes analysis of the convergence properties for a network with dynamic graph, investigating the effect of non-uniform activation of the agents, and extending the convergence rate analysis to the other second order asynchronous methods.
 
\bibliographystyle{plain}
\bibliography{citeNNarxiv}
\end{document}